\documentclass[12pt,reqno]{amsart}

\usepackage[T1]{fontenc}
\pdfoutput=1
\usepackage{mathptmx,microtype,hyperref,graphicx,dsfont}
\usepackage{amsthm,amsmath,amssymb} 
\usepackage[foot]{amsaddr}
\usepackage{enumitem}
\usepackage[nameinlink,capitalise]{cleveref}
\usepackage[font=footnotesize,margin=1.6cm]{caption}
\usepackage{cite}

\newtheorem{theorem}{Theorem}[section]
\newtheorem{lemma}[theorem]{Lemma}
\newtheorem{conjecture}[theorem]{Conjecture}
\newtheorem{openproblem}[theorem]{Open Problem}

\theoremstyle{definition}

\newcounter{example}

\newcounter{open}

\newcounter{figno}

\newcounter{tableno}

\numberwithin{equation}{section}
\numberwithin{figure}{section}
\numberwithin{table}{section}
\numberwithin{example}{section}

\newcommand{\prob}[1]{\mathbb{P}\left(#1\right)}

\newcommand{\Go}{G_{\text{\rm open}}}
\newcommand{\Eo}{E_{\text{\rm open}}}

\newcommand{\une}{\leftrightarrow}
\newcommand{\ore}{\rightarrow}
\newcommand{\er}{\rightarrow}
\newcommand{\el}{\leftarrow}
\newcommand{\pr}{p_{\text{\rm right}}}
\newcommand{\pl}{p_{\text{\rm left}}}
\newcommand{\po}{p_{\text{\rm open}}}

\newcommand{\ER}{Erd\H{o}s--R\'{e}nyi }

\newcommand{\lto}{\leftarrow}
\newcommand{\lrto}{\leftrightarrow}
\newcommand{\lr}{\leftrightarrow}

\newcommand{\cA}{{\mathcal A}}

\newcommand{\cO}{\mathcal O}

\newcommand{\cE}{\mathcal E}

\newcommand{\cI}{\mathcal I}

\newcommand{\E}{{\mathbb E}}

\author[J. Gravner]{Janko Gravner}
\address{University of California, Davis, Department of Mathematics}
\email{gravner@math.ucdavis.edu}

\author[B. Kolesnik]{Brett Kolesnik}
\address{University of Oxford, Department of Statistics}
\email{brett.kolesnik@stats.ox.ac.uk}


\begin{document}

\title[Transitive closure in a polluted environment]
{Transitive closure in a polluted environment}

\begin{abstract}
We introduce and study a new percolation model, inspired
by recent works on jigsaw percolation, graph bootstrap percolation, 
and percolation in polluted environments. 
Start with an oriented graph $G_0$ of initially occupied edges 
on $n$ vertices, and
iteratively occupy additional (oriented) edges by transitivity, with the 
constraint that only open edges in a certain random set 
can ever be occupied. 
All other edges are closed, creating a set of obstacles for the 
spread of occupied edges. 
When $G_0$ is an unoriented linear graph, 
and leftward and rightward edges are open independently 
with possibly different probabilities, we identify 
three regimes in which the 
set of eventually occupied edges is 
either all open edges, the majority of open 
edges in one direction, or only a very small proportion of 
all open edges. In the more general setting where 
$G_0$ is a connected unoriented graph of bounded 
degree, 
we show that the transition between 
sparse and full occupation of open 
edges occurs when the probability  
of open edges
is $(\log n)^{-1/2+o(1)}$. 
We conclude with several conjectures and  
open problems. 
\end{abstract}

\maketitle

\section{Introduction}\label{sec-intro}

Suppose that we have $n$ logical statements, each represented 
by a vertex of a graph $V$, and that they are all equivalent, 
but we are not aware of this fact. 
The initial information
consists of some implications, and is realized as
an oriented subgraph $G_0=(V,E_0)$. We then 
try to logically complete the knowledge by transitivity. 
However, a capricious ``censor'' allows only 
certain conclusions to be made, represented 
by  open edges. A natural question is whether 
a substantial proportion of uncensored  knowledge can 
be obtained by this transitive closure process. 

Another application is as follows. Suppose we want to compute the product $a_1a_2\cdots a_{n-1}$ in a 
noncommutative group. However, some of the subproducts, and their inverses, are 
not allowed to be computed. Can the product still be computed? If all 
$a_i$ and $a_i^{-1}$ are initially known, then $G_0$ is the 
unoriented
linear graph 
$L_{n}$ on the points $[n]=\{1,2,\ldots,n\}$
with edges between nearest neighbors. Rightward edges in $G_0$ represent the $a_i$, 
leftward edges in $G_0$ represent their inverses $a_i^{-1}$,
and vertices in $G_0$ are positions for multiplication brackets. 
Longer edges between vertices in $[n]$ represent other 
elements in the group. 

We now introduce our dynamics more formally. All of our 
graphs will have a fixed vertex set $V$ of $n$ points. 
In many contexts, it is convenient to take 
$V=[n]$. 
We denote oriented and unoriented edges using the notations 
$i\ore j$ and $i\une j$. Throughout we identify unoriented edges 
with two edges in both directions.
As our focus is transitive closure, it is convenient to adopt the notation 
$i\ore j\ore k$ for the pair of oriented edges  $i\ore j$ and $j\ore k$. 
Likewise, we make use of similar 
abbreviations, such as $i\leftarrow j\to k$ 
and $i\to j\leftarrow  k$. 

We consider an evolving sequence $G_t=(V,E_t)$, 
$t=0,1,\ldots$ of graphs, with the set of {\it occupied edges\/}
$E_t\subset V\times V$ by time $t$ nondecreasing in time, 
that is, $E_t\subset E_{t+1}$. We denote 
the set of  eventually occupied edges
by 
$E_\infty=\bigcup_{t\ge 0} E_t$, and 
put $G_\infty=(V,E_\infty)$.
More specifically, our {\it transitive closure\/}
dynamics, once initialized by some $G_0=(V,E_0)$, are 
governed by another graph $\Go=(V,\Eo)$, where $\Eo\subset (V\times V)\setminus E_0$ are
{\it open\/} edges. 
Note, in particular, that 
the sets of initially occupied and open edges are disjoint, 
$\Eo\cap E_0=\emptyset$. 
The edges in $(V\times V)\setminus (\Eo\cup E_0)$ 
are called {\it closed\/}. 
The status of self-loops $i\une i$
will be  irrelevant, but for concreteness, we  
assume they are all closed. 
The dynamics evolve as follows: 
given the set of occupied 
edges $E_t$ at time $t$, we let 
\begin{equation}\label{E_Et}
E_{t+1}=E_t\cup \{i\to j\in \Eo:i\to k\to j\in E_t,\mbox{ for some }k\in V\}. 
\end{equation}
In words, 
an open edge $i\to j$ becomes occupied at time $t+1$
if there is a series of two occupied edges
$i\to k\to j$ at time $t$. 

If $G_0$ is strongly connected and all edges not initially occupied are open, then it is clear that $G_\infty$ is a  
complete graph. Thus it is natural to ask what happens when 
some --- most, in our case --- edges are closed and thus 
unable to ever become occupied. 
In this introduction, 
we will assume that $G_0$ is a deterministic 
connected unoriented graph. 
In general, when $G_0$ does not have extra structure, $\Go$ will be
the 
oriented Erd\H{o}s--R\'{e}nyi
graph with edge probability $\po>0$.
(To be more precise, 
this is 
a slightly modified 
version in which each oriented edge
{\it not} in $E_0$ is open with probability $\po>0$ and closed otherwise.)
We note here that the case when $\Go$ is 
unoriented is easier, and also results like
Theorem~\ref{tc2} below are not possible.

Some of our results are 
 concerned with the specific case 
when $G_0=L_n$ is the unoriented linear graph
with edges 
$1\lr 2\lr\cdots\lr n$, 
and it is in this case
that we may assign 
different probabilities $\pl>0$ and $\pr>0$ to leftward 
and rightward open edges. 
The probabilities $\po$, $\pl$ and $\pr$ may depend on $n$, however, we 
suppress this notationally whenever the dependence is clear in context.

We say that a 
subset $V'\subset V$ is \emph{saturated} at time $t$ if all 
open edges in $V'\times V'$ are occupied at this time. When we do 
not make a reference to time, we mean $t=\infty$, that is,
$V'$ is saturated eventually. For an edge $i\ore j$, 
we define its {\it length\/} as the number of edges on the shortest 
oriented
path in the 
graph $G_0$ from $i$ to $j$ (or $\infty$ if no such path exists). 
For instance, 
when $G_0=L_n$,
the length of $i\ore j$ is simply $|i-j|$. 

Our first result is
for general initial graphs of  bounded degree. 

Recall that a sequence of events $A_n$ hold 
{\it asymptotically almost surely,} abbreviated a.a.s., if 
their probabilities converge to 1.

\begin{theorem}
\label{tc1} Assume that $G_0=(V,E_0)$ is a 
connected unoriented graph on $V=[n]$ with vertex degrees 
bounded by a constant $D$, and 
that open 
(oriented) edges are chosen independently (from amongst 
those not in $E_0$) with probability $\po$.
Fix a constant $\alpha>0$. 
Then there exist constants $c\in(0,\infty)$ depending on 
$D$ and $\alpha$,  and $C\in(0,\infty)$ depending only on $D$, 
so that the following statements hold. 
\begin{itemize}
\item[{\rm (1)}] When $\po<c\frac 1{\sqrt{\log n}}$, a.a.s.\ $E_\infty$ contains no 
edge longer than $\alpha\log n$.
\item[{\rm (2)}] 
When $\po>C\frac{\log\log n}{\sqrt{\log n}}$, a.a.s.~saturation occurs,  $E_\infty=E_0\cup \Eo$. 
\end{itemize}
\end{theorem}

We remark that the identical 
result (with easier proof) holds under the assumption 
that $\Go$ is the unoriented \ER graph with probability  $\po$ 
of open edges.

\begin{figure}
\centering
\includegraphics[trim=0cm 0cm 0cm 0cm, clip, width=.3\textwidth]{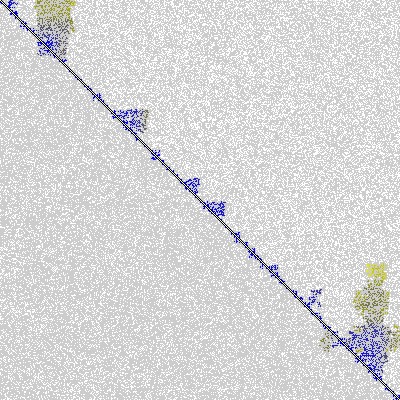}
\hskip0.1cm
\includegraphics[trim=0cm 0cm 0cm 0cm, clip, width=.3\textwidth]{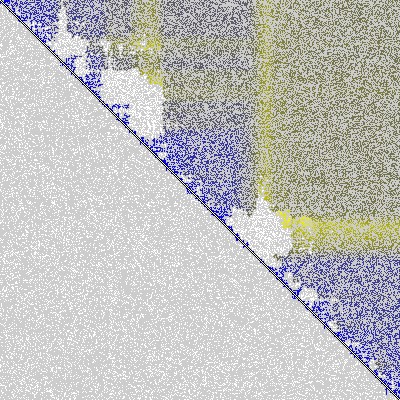}
\hskip0.1cm
\includegraphics[trim=0cm 0cm 0cm 0cm, clip, width=.3\textwidth]{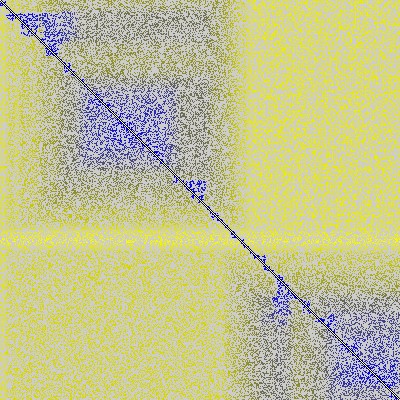}
\caption{Illustration of the three regimes in Theorem~\ref{tc2} when 
$n=300$: subcritical (left, with $\pl=0.24$, 
$\pr=0.36$), intermediate (middle, with $\pl=0.2$, 
$\pr=0.4$; note the 
non-monotone fashion in which edges are occupied), and 
supercritical (right, with $\pl=\pr=0.35$; note the 
nucleation). 
The dynamics are represented as the evolution of the adjacency matrix,
with edges exhibited as sites in the square. Initially occupied
sites next to the diagonal are black, closed sites are grey and
open sites are white. After the transitive closure process is complete,
the initially white sites that become occupied are colored according
to the time of occupation, from blue (the earliest) to yellow (the
latest).
}
\label{fig-tra}
\end{figure}

Our next theorem establishes three regimes in
the case of the unoriented linear graph. 

\begin{theorem}  \label{tc2}
Assume that $G_0=L_n$ is the unoriented linear 
graph on $V=[n]$
with edge set $E_0$ consisting of all edges $1\lr 2\lr\cdots\lr n$. 
Suppose that open leftward and rightward edges are chosen independently (from 
amongst those not in $E_0$) with probabilities $\pl$ and $\pr$. Fix a constant $\alpha>0$. 
Then there exist constants $c\in (0,\infty)$ 
and $A\in (0,1)$ depending on $\alpha$, and a constant $C\in(0,\infty)$, 
so that the following three statements hold.

\begin{itemize}
\item[{\rm (1)}] 
When $\max\{\pl,\pr\} < c\frac 1{\sqrt{\log n}}$,
a.a.s.\ $E_\infty$ contains no 
edge longer than $\alpha\log n$.
\item[{\rm (2)}] When $\pl<c\frac 1{\sqrt{\log n}}$ and 
$\pr>A$,  a.a.s.\ $E_\infty$ contains 
all open rightward edges 
longer than $\alpha\log n$, but no such  
leftward edge.
\item[{\rm (3)}] When $\min\{\pl,\pr\}> C\frac{\log\log n}{\sqrt{\log n}}$,
a.a.s.~saturation occurs,  $E_\infty=E_0\cup \Eo$. 
\end{itemize}
\end{theorem}

While it is not realistic to expect that simple
simulations can distinguish between $\sqrt{\log n}$ and 
a constant, we illustrate the three regimes 
guaranteed by Theorem~\ref{tc2} in Fig.~\ref{fig-tra}.

It appears to be a challenge to extend the subcritical case (1), due to interactions
between leftward
and rightward edges.

For comparison, we also state the following result for the oriented 
linear graph $G_0=L_n^\to$, where 
the edges $1\er 2\er\cdots\er n$ are initially occupied, all other 
rightward edges are open with 
some probability $\pr>0$, and all leftward edges are closed ($\pl=0$). For reasons 
that will become clear in Section~\ref{sec-catalan}, 
we introduce the term {\it Catalan percolation} for this 
instance of our
process. 
In contrast with the unoriented case $G_0=L_n$, where saturation occurs 
at a probability $(\log n)^{-1/2+o(1)}$ of open edges, in this oriented case  
the probability must be very close to $1$ for saturation. 
Part (3) of the following theorem calculates the asymptotics
of this probability. Parts (1) and (2) show that the threshold 
for ``near-saturation'' is of constant order, bounded away from 
0 and 1. 

\begin{theorem}\label{tc3}
Assume 
$G_0=L_n^\to$ is the oriented linear graph
on $V=[n]$ with edge set $E_0$ consisting of all edges $1\er 2\er\cdots\er n$.
Suppose that open leftward and rightward edges are chosen independently (from 
amongst those not in $E_0$) with probabilities  
$\pl=0$
and $\pr=p$. Then the following statements hold. 
\begin{itemize}
\item[{\rm (1)}] For any constant $p<1/4$, a.a.s.\ $E_\infty$ contains no edge
longer than $C\log n$, for some constant $C=C(p)$. 
\item[{\rm (2)}] There is a constant $p_u<1$ so that for all constants $p\in(p_u,1)$, 
a.a.s.\ $E_\infty$ contains all open edges of length $C'\log n$, for some 
constant $C'=C'(p)$.
\item[{\rm (3)}] If $p=1-\alpha n^{-1/2}$, for some constant $\alpha>0$, then 
the probability of saturation ($E_\infty$ contains all open rightward edges) approaches 
$e^{-\alpha^2}$ as $n\to\infty$. 
\end{itemize}
\end{theorem}  

To put our results in the context of the literature, 
let us note that the algorithm by which edges become 
occupied according to \eqref{E_Et}
is related to {\it graph bootstrap percolation}
\cite{Bol, BBM, BK20} 
(in particular, see the discussion following Problem 6
in \cite{BBM}), but in its analysis, as well as in its 
modeling of increasing partial knowledge, it more 
closely resembles {\it jigsaw percolation} \cite{BCDS, GS, BRSS, CKM}. 
As is clear from Fig.~\ref{fig-tra}, the supercritical regime 
in this process 
is characterized by {\it nucleation\/}. That is, local events create a network of occupied edges large enough to
be unstoppable: with high probability it continues to 
occupy edges on its boundary until, finally,  no
open edges remain unoccupied. Perhaps the most well-known nucleation 
process is {\it bootstrap percolation} \cite{PRK75,CLR79}, which has been studied 
in great detail and yielded numerous deep and surprising 
results. Here we only mention three milestone papers \cite{AL, Hol, BBDM}. 
Due to the fundamental significance of this model, methods
and concepts which have resulted from its study are likely useful in the analysis of any 
nucleation process, and ours is no exception. We should also
mention that the polluted version of bootstrap 
percolation has also been investigated \cite{GM, GH, GHS},  however, with
the emphasis on random initial states and thus on results of a different flavor.

By contrast, Catalan percolation and the related intermediate regime
have ties to classical results on random graphs: we establish the
constant-order threshold for the formation of a giant component, while saturation
is avoided primarily by the appearance of the shortest closed edges that
can prevent it, which is analogous to the containment problem
\cite{RG}.
Finally, we also mention the work of Karp \cite{Karp}, which studies strongly connected
components in directed random graphs and 
the time to complete the transitive closure
process.

\subsection{Outline}

Most of the rest of this article is devoted to proofs of the above 
three theorems. We in fact prove a bit more, and so
some of the statements will be given in a more general form. 
Due to the connections between the parts of these results, they are proved
in a different order than stated above: 
In Section~\ref{sec-general-subcr} we prove the subcritical result Theorem~\ref{tc1} (1) 
for bounded-degree initial graphs $G_0$. 
This implies Theorem~\ref{tc2} (1) for the linear graph $G_0=L_n$. 
Theorem~\ref{tc3} for Catalan percolation is proved in Section~\ref{sec-catalan}. Theorem~\ref{tc3} (2)
is used to establish  
the statement about rightward edges in the intermediate result Theorem~\ref{tc2} (2). 
The remainder of this result, concerning leftward edges, 
is dealt with in Section~\ref{sec-tricky}.
Section~\ref{sec-supercr} establishes the supercritical result 
Theorem~\ref{tc1} (2), which implies 
Theorem~\ref{tc2} (3). 
We conclude with  Section~\ref{sec-open}, which contains 
a selection of open problems.

\subsection{Notation}
We use standard asymptotic notation throughout, such 
as $f\ll g$ and $f=o(g)$ if $f(n)/g(n)\to0$ as $n\to\infty$.
In particular, $o(1)$ denotes
a function $f$ such that $f\ll1$. 

\subsection{Acknowledgments}

We thank the referees for their comments that 
helped improve this work. 
JG was partially supported by the NSF grant DMS-1513340 and the Slovenian Research Agency
research program P1-0285. BK was 
affiliated with the University of California, Berkeley
while the research 
and writing of this article took place, and was 
partially supported by an NSERC Postdoctoral Fellowship.

\section{Subcritical regime for bounded-degree initial graphs}
\label{sec-general-subcr}

In this section, we prove Theorem~\ref{tc1} (1), which recall implies Theorem~\ref{tc2} (1). 

We begin with a series of deterministic Lemmas~\ref{connected-necessary}, \ref{first-connection} 
and \ref{Ie-lemma}
that provide a necessary condition for an edge $e$
to become occupied. Roughly speaking, Lemma~\ref{Ie-lemma} implies the existence
of a set $I_e$, larger than the length of $e$, with the property that each $v\in I_e$ is 
the base of a special type of 
oriented triangle, called a {\it horn}. The most crucial property of a horn is that it contains
at least one open edge.  
Horns will also play a key role in the subsequent Sections~\ref{sec-tricky} and \ref{sec-supercr}
(intermediate and supercritical regimes).
The next result Lemma~\ref{AL-property} establishes an 
Aizenman--Lebowitz \cite{AL} type property for the sizes of sets $I_e$ of eventually occupied 
edges. These results, together with a simple search algorithm Lemma~\ref{L_BK} 
for edge-disjoint horns based in a set $I_e$, imply the main result Theorem~\ref{tc1} (1).

 \begin{lemma}\label{connected-necessary}
 Assume that $i\ore j\in E_\infty$. Then there exists 
 an oriented path from $i$ to $j$ in $G_0$.
 \end{lemma}

 \begin{proof}
The proof is by induction on the time of occupation. 
The statement is immediate for edges in $E_0$. 
For an edge $i\to j\in E_{t+1}$, 
there are edges $i\to w\to j\in E_t$, and so by induction, 
oriented paths from $i$ to $w$ and from $w$ to $j$. 
Concatenating these paths, we obtain an oriented path from $i$ to $j$
(after deleting any loops).  
 \end{proof}
 
For sets $A,B\subset V$ we say that an edge $e$ is an edge from $A$ to $B$ if $e=a\to b$
for some $a\in A$ and $b\in B$. 

\begin{lemma}\label{first-connection} Assume $V_1\subset V$. 
Assume $E_\infty\setminus E_0$ contains an edge 
from $V_1$ to $V_2=V\setminus V_1$.
Then there exist vertices $v_1\in V_1$, $v_2\in V_2$ 
and $w\in V$ so that 
$v_1\to w\to v_2\in E_0\cup\Eo$,  
$v_1\to v_2\in \Eo$
and either 
(1) $w\in V_1$ and $w\to v_2\in E_0$, 
or else, (2) $w\in V_2$ and $v_1\to w\in E_0$.
 \end{lemma}

 \begin{proof}
Let $t\ge 1$ be the first time that an edge $v_1\to v_2\in \Eo$ from $V_1$ to $V_2$
becomes occupied. 
Then, for some $w\in V$,   
$v_1\to w\to v_2\in E_{t-1}\subset E_0\cup \Eo$. 
Then, by the minimality of $t$, either (1) 
$w\in V_1$ and $w\to v_2\in E_0$, or else, (2) $w\in V_2$ and $v_1\to w\in E_0$.
 \end{proof}
 
For edges $e\in E_\infty$, we define 
$$\cI_e=\{V_0\subset V:\text{the subgraph of 
$(V,E_0\cup \Eo)$ 
induced by $V_0$ makes $e$ occupied}\}. 
$$
With each such $e$, we associate an arbitrary $I_e\in \cI_e$ of  
minimal cardinality. 

Our next lemma shows that if $v\in I_e$, then $v$ is in a triangle of 
a certain type. 

Let $K\subset V$. A vector $(v,x,y)\in K^3$ is a {\it horn} in $K$ if 
any of the following
conditions are satisfied (see Fig.~\ref{fig-horns}):
\begin{enumerate}
\item $x\to v\to y\in E_0$ and $x\to y\in \Eo$, 
\item $v\to y\in\Eo$, $v\to x\in E_0$ and $x\to y\in E_0\cup \Eo$,
\end{enumerate} 
or, in the opposite  orientation,
\begin{enumerate}
\item[3.] $x\lto v\lto y\in E_0$ and $x\lto y\in \Eo$, 
\item[4.] $v\lto y\in\Eo$, $v\lto x\in E_0$ and $x\lto y\in E_0\cup \Eo$. 
\end{enumerate} 

In all cases, we call $v$ the {\it base} 
and $y$ the {\it tip} 
of the horn.  

\begin{figure}[h!]
\centering
\includegraphics{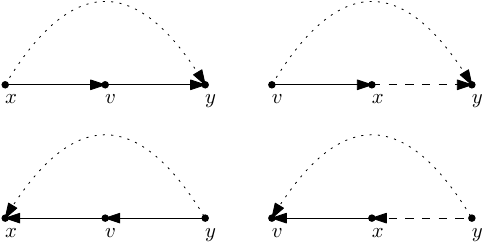}
\caption{The four ways $v$ can be the base and $y$ the tip of a horn. 
Initially occupied edges in $E_0$ are represented by solid arrows, open edges in $\Eo$
by dotted arrows, and edges in $E_0\cup \Eo$ by dashed arrows.
}
\label{fig-horns}
\end{figure}

\begin{lemma} \label{Ie-lemma} Assume 
$e\in E_\infty\setminus E_0$, and $v\in I_e$. 
Then $v$ is the base of a horn in $I_e$.
\end{lemma}

\begin{proof}
Replace $V$ with $I_e$, and replace   
$G_0$ and $\Go$ with their subgraphs induced by $I_e$, so that 
$E_0$ and $\Eo$ now only contain edges between vertices in $I_e$. 

By the minimality of $I_e$, we have $I_e\setminus\{v\}\notin {\mathcal I}_e$. 
Therefore, informally,  
some edge
in $\Eo$ would not get occupied without the ``help'' of $v$.
Let $e'$ be some such edge in $E_t$, where $t\ge1$ is the 
first time such an edge becomes occupied. There are two cases: 

(a) If $v$ is an endpoint of $e'$, then either $e'=v\to y$ or $e'=v\lto y$
for some $y\in I_e\setminus\{v\}$. Therefore, by the choice of $t$, it follows in these cases
that $v\to x\in E_0$ and $x\to y\in E_{t-1}$ or  
$v\lto x\in E_0$ and $x\lto y\in E_{t-1}$ for some $x\in I_e\setminus\{v,y\}$, 
and hence that $v$ is the base of a horn in $I_e$. 

(b) On the other hand, if $e'=x\to y$
for some $x,y\notin I_e\setminus\{v\}$, then $x\to v\to y\in E_{t-1}$. 
If $x\to v\to y\in E_0$ then it is immediate that $v$ is the base of a horn 
in $I_e$. 
Otherwise, if 
$x\to v\notin E_0$
use Lemma~\ref{first-connection} with $V_2=\{v\}$ to see that 
$v$ is the base of a horn in $I_e$.
Similarly, if 
$v\to y\notin E_0$, 
use $V_1=\{v\}$ instead. 
 \end{proof}

Next, we state a crucial property for establishing 
the subcritical regime of our iterative growth process. This 
property, first formulated by Aizenman and Lebowitz \cite{AL}
in the context of bootstrap percolation, implies that the transitive closure 
dynamics create sets, with certain internal properties, 
of sizes on all scales 
smaller than the longest length of an occupied edge. The proof 
hinges on a {\it slowed-down\/} version of the dynamics, whereby 
at each time step we occupy a single open edge (that can be occupied by a transitive step).  
This edge is chosen arbitrarily from the available edges until no such edge exists. 
The monotonicity of the original process implies that any 
slowed-down version produces the same set of eventually occupied edges.
 
Recall that the length of an edge $e=i\to j$ is the number of edges in the shortest oriented path
($\infty$ if no such path exists) 
from $i$ to $j$ in $G_0$. 
 
 \begin{lemma}\label{AL-property}
Assume $e_0\in E_\infty$ has length $\ell$.
Then for 
every integer $k\in [1,\ell]$, there is an edge $e$ with 
$|I_e|\in [k+1, 2k]$. 
\end{lemma}

\begin{proof}
Remove all edges from $E_0\cup \Eo$ 
besides those between vertices of $I_{e_0}$, 
and then consider the slowed-down process, terminated once $e_0$ is occupied.   
If at some step an edge $e=x\to y\in \Eo$ is 
occupied by ``parent'' edges $e'=x\to z$ and $e''=z\to y$, then 
$I_{e'}\cup I_{e''}\in \cI_e$ and so 
$|I_{e}|\le |I_{e'}|+|I_{e''}|$. Therefore, at each step of the
slowed-down process, the maximal cardinality of $|I_e|$, over all edges $e$ 
occupied thus far, at most doubles. As this 
maximum starts at $2$ and ends at $|I_{e_0}|$, 
the claim follows, noting that $|I_{e_0}|\ge\ell+1$ by 
Lemma~\ref{connected-necessary} (and since, by assumption, $e_0$  has length $\ell$). 
\end{proof}

For the rest of this section, assume 
that the in-degrees and out-degrees of the initial 
graph $G_0$ are bounded by an integer
$D\ge 1$. 

In this setting, we collect one more lemma before turning
to the main result of this section. 

\begin{lemma}\label{L_BK}
Suppose that $K\subset V$ is such that all $v\in K$ are bases of horns in $K$. 
Then there is a set $K_0\subset K$ of size at least $|K|/(9D)$
so that horns (in $K$) for each $v\in K_0$ can be chosen so that their
edge-sets are pairwise disjoint. 
\end{lemma}

\begin{proof}
This can proved by a simple search algorithm. Order the vertices of $K$ arbitrarily. 
Start with $K_0=\emptyset$ and another set $U=\emptyset$ of used vertices,
and enlarge them as follows. 
Let $d_0$ be the graph distance in $G_0$. 
In each step, find the first vertex $v$ such that $d_0(v,U)>1$
and a horn $(v,x,y)$ in $K$. 
Note that $x\notin U$, but  the tip $y$ could 
possibly be in $U$. Add $v$ to $K_0$ and all of $v,x,y$ to $U$. 
The proof now follows by observing that, after $t$ steps, 
there are at most $3(1+2D)t\le 9Dt$ vertices $u$ such that $d_0(u,U)\le1$, and 
that any horn based at some $v'$ with $d_0(v',U)>1$ does not involve any edges {\it between}
vertices in $U$. 
\end{proof}

Finally, we prove Theorem~\ref{tc1} (1), 
which we state below in a stronger form, as we do not need 
to assume that the initial graph is unoriented. 

\begin{theorem}\label{tg-subcr}
Assume that $G_0=(V,E_0)$ is a connected 
graph on $V=[n]$ with in-degrees and out-degrees
bounded by a constant $D$. 
Fix a constant
$\alpha>0$. 
Then there exists a constant $c=c(\alpha)>0$, 
so that if open 
(oriented) edges are chosen independently (from amongst 
those not in $E_0$) with probability $\po<c/\sqrt{\log n}$, then 
$$\prob{\text{some edge of length at least $\alpha\log n$ becomes occupied}}\to 0,$$
as $n\to\infty$. 
\end{theorem}

\begin{proof} 
The idea is to show that an occupied edge
of length $\ell=\lfloor\alpha\log n\rfloor$ or longer
implies the existence of many edge-disjoint horns, and so, many open edges. 
To this end, consider the unoriented graph 
$\widetilde G_0$ obtained 
from $G_0$ by ignoring orientation, i.e., $i\une j$
is an edge of $\widetilde G_0$ if either $i\el j\in G_0$ 
or $i\er j\in G_0$. 

First note that, for any fixed (deterministic) set
$K\subset V$ of size $k$,
\begin{equation}\label{E_horn}
\prob{v\text{ is the base of a horn in }K}
\le 
2(2D)^2\po+2Dk\po^2,
\end{equation}
where (see Fig.~\ref{fig-horns}) the first term bounds the event that $v$ is the base of a horn
in $K$ with only one open edge, and the other term
bounds the case of other types of horns involving two open edges. 
Next, by Lemma~\ref{L_BK} and the 
van den Berg--Kesten inequality \cite{BK85}  
we claim that, for any such $K$,
the probability that all vertices in $K$ are bases of horns in $K$ is at most 
\[
2^{k} (8D^2\po+2Dk\po^2)^{k/(9D)}. 
\] 
Indeed, the number of ways to select
$k/(9D)$ vertices in $K$  is ${k\choose k/(9D)}\le 2^{k}$. 
Further, any given given $k/(9D)$ vertices in $K$ 
are bases of edge-disjoint horns in $K$
with probability at most 
the upper bound in  \eqref{E_horn} to the power 
$k/(9D)$. 

Next, we claim that if an edge of length $\ell$ is occupied, then 
there is some $\widetilde G_0$-connected set $K$ of size
$k\in[\ell/2,\ell]$ such that all vertices $v\in K$ are bases of horns 
in $K$. To see this, note that if some $e_0$ of length at least $\ell$
is occupied then by Lemma~\ref{AL-property} 
there is an edge $e\in E_\infty$ with 
$|I_e|\in [\ell/2,\ell]$. By Lemma~\ref{connected-necessary} (with $k\in\lfloor\ell/2\rfloor$) $I_e$ is 
$\widetilde G_0$-connected, and 
by Lemma~\ref{Ie-lemma} 
every $v\in I_e$ is the base of a horn in $I_e$, giving the claim.

Finally, by e.g.\ Lemma 3.5 in \cite{GS}, 
 the number of 
$\widetilde G_0$-connected 
subsets of $V$ of size $k$ 
containing a given vertex is at most $(6D)^k$.
Putting all of the above together, a union bound yields 
\begin{equation*}
\begin{aligned}
&\prob{\text{some edge of length at least $\ell$ becomes occupied}}\\
&\le n\sum_{k=\ell/2}^\ell (6D)^k\cdot 2^k(8D^2\po+2Dk\po^2)^{k/(9D)}\\
&\le n\ell (12D)^\ell(8D^2\po+2D\ell\po^2)^{\ell/(18D)}\\
&=\alpha n\log{n}\left[12D(8cD^2/\sqrt{\log n}+2D\alpha c^2)^{1/(18D)}\right]^{\lfloor\alpha\log n\rfloor}
\ll 1
\end{aligned}
\end{equation*}
for all sufficiently small $c>0$.
 \end{proof}

\section{Catalan percolation}\label{sec-catalan}

In this section, we focus on Catalan percolation, which recall is 
the transitive closure process on the oriented linear graph $G_0=L_n^\to$, 
consisting of the initially occupied edges $1\er 2\er\cdots \er n$, 
in the case that all leftward edges are closed ($\pl=0$) and 
all rightward edges (of length at least 2) are open independently 
with some probability $\pr=p$. 
Observe that in this setting, the length of an edge $i\to j$ is simply $j-i$. 

We first prove Theorem~\ref{tc3} (1) and (2), stated as Lemmas~\ref{catalan-sub}
and \ref{catalan-sup} below, which together establish that the threshold 
for the occupation of ``long'' edges is of constant order, bounded between 0 and 1. 
The proof of Lemma~\ref{catalan-sub} is a simple combinatorial argument, but reveals
a connection with the Catalan numbers, which is the reason for the name of the process. 
On the other hand, Lemma~\ref{catalan-sup} is proved by noticing that a certain restriction of the dynamics 
can be described using {\it oriented percolation}  \cite{Dur}. 
The probability of saturation (occupation of all open edges), 
Theorem~\ref{tc3} (3), is discussed afterwards at the
end of this section. 

\begin{lemma}\label{catalan-sub}
For any constant $p<1/4$, there exists a constant $C=C(p)$ so that 
a.a.s.~all edges in $E_\infty$ have length at most $C\log n$.
\end{lemma}

\begin{proof}
Assume $e$ is an oriented edge of length $\ell$. Let 
$\cE_e$ be 
the set of all inclusion-minimal 
sets of open edges (including $e$) that, together with edges in $E_0$, 
make $e$ occupied. 
By induction, it is easy to see that any $A\in\cE_e$ is of size $|A|=\ell-1$, 
and moreover $|\cE_e|=C_\ell$, the $\ell$th Catalan number. 
One way to see this is to consider computing a 
product of $a_1a_2\cdots a_{\ell}$ as described in Section~\ref{sec-intro}. 
Then each element in $\cE_e$ corresponds with a way of  
parenthesizing the product. 
Since $C_\ell\le 4^\ell$, it follows that 
\begin{multline*}
\prob{\text{an edge of length at least $C\log{n}$ becomes occupied}}\\
\le n^2 p^{-1}(4p)^{C\log n}\ll1
\end{multline*}
for all $C>-2/\log(4p)$. 
\end{proof}

In preparation for the proof of the next result, 
it will be useful to view the growth dynamics on $[n]^2$.
As such, we will often use the terms ``edge''
and ``site'' interchangeably when referring to an edge $i\to j$ and its corresponding 
site $(i,j)$. 
As in Fig.~\ref{fig-tra}, the site $(i,j)$ for an edge $i\to j$ is positioned in $[n]^2$ as
in the adjacency matrix (with the $y$-axis oriented downwards). 
The 
initially occupied sites are those in $\{(i,i+1):i=1,\ldots, n-1\}$ 
and only  the sites above this diagonal 
may ever become occupied. 
Open sites $(i,j)\in[n]^2$ become occupied once  there are occupied sites 
$(i,k)$ and $(k,j)$, for some $i<k<j$. 

\begin{figure}[h!]
\centering
\includegraphics[scale=0.8]{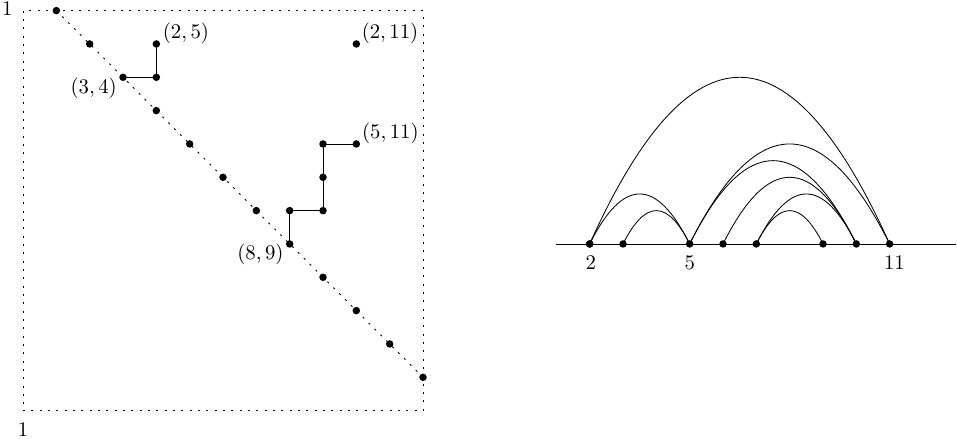}
\caption{
{\it At left:} The open edge $(2,11)$ becomes occupied
due to a pair of oriented paths of open sites, from $(3,4)$ to $(2,5)$ and 
from $(8,9)$ to $(5,11)$. 
{\it At right:} 
The same occupation process whereby edges
are represented as usual. 
Note that, until the very last step, 
each transitive step involves at least one initially occupied edge.
}
\label{fig-OP}
\end{figure}

The advantage of this point of view 
is its connection with oriented percolation. See Fig.~\ref{fig-OP}. 
It is easy to see (by induction) that an open site $(i,j)$ becomes occupied
if there is an {\it oriented percolation path} 
(moving one unit up or to the right in each step) 
along open sites, starting 
from some (initially occupied) site on the diagonal to $(i,j)$. 
Indeed, moving up from a site $(x,y)$ corresponds to occupying 
$(x-1)\to y$ due to $x\to y$ and $(x-1)\to x$ being occupied, 
and moving to the right corresponds to occupying 
$x\to (y+1)$ due to $x\to y$ and $y\to (y+1)$ being occupied.
This connection plays a crucial role 
in the proof of Lemma~\ref{catalan-sup} below, which shows 
by a contour/duality argument (standard for oriented percolation)
that with high probability
all open edges of $i\to j$ of length at least $C\log n$ become occupied  
due to  a pair of 
oriented paths from (possibly different starting points on) the diagonal to sites 
$(i,k)$ and $(k,j)$, for some $i<k<j$.

\begin{lemma}\label{catalan-sup}
There exists a constant $p_u<1$ so that 
the following holds. If $p>p_u$, then there exists a constant $C=C(p)$ so that 
a.a.s.~$E_\infty$ contains every open edge of length at least $C\log n$.
\end{lemma}

\begin{proof} We divide the proof into three steps.
Recall that, as discussed above, 
we identify edges $i\ore j$ with sites $(i,j)\in [n]^2$.  

\noindent{\bf Step 1.} Assume that, for $i<j$, there 
exists an oriented percolation path of open or initially occupied sites
connecting a site in $G_0$ to $i\er j$. Then $i\ore j\in E_\infty$.

As already sketched above, the proof of Step 1 is a simple induction argument on the 
length $\ell$ of $i\ore j$. The claim holds when 
$\ell=1$ as those edges are in $E_0$. 
Otherwise, for $\ell>1$, the oriented percolation path from the diagonal to $(i,j)$
must visit either $(i,j-1)$ or $(i+1,j)$ (i.e., either the site to the left or below $(i,j)$
in the adjacency matrix)
before reaching $(i,j)$. So, by the induction hypothesis, 
either $i\ore (j-1)$ or $(i+1) \ore j$ becomes  
occupied. Then, since $(j-1)\ore j$ and
$i\ore (i+1)$ are initially occupied, the claim follows. 

\noindent{\bf Step 2.} Fix an $\ell>1$.  
Let $F_\ell$ be the event that
strictly more than $\ell/2$ sites on 
$L=\{(1,i):2\le i\le \ell+1\}$ 
are connected to $G_0$ through oriented percolation paths. 
Then, for $p>1-2^{-32}$,  we claim that 
$$P(F_\ell^c)\le 2\cdot 8^\ell (1-p)^{\ell/8}.
$$ 

This follows by a typical 
contour argument
 (see e.g.\ \cite{Dur} Section~10). 
 
 Choose any subset $S$
of $L$ of size at least $\ell/2$, and assume that 
$S$ is exactly the set of sites 
that are not connected to $G_0$ by oriented 
percolation paths. Write 
$S=S_1\cup \cdots\cup S_m$, 
where $S_i$ are non-adjacent intervals. Then, by a standard duality argument, 
there exist disjoint paths 
$\pi_i:x_0^{(i)},\ldots, x_{t_i}^{(i)}$, $i=1,\ldots,m$,  
such that (1) 
$||x_{j}^{(i)}-x_{j-1}^{(i)}||_\infty=1$ for $j=1,\dots,t_i$, (2) $x_0^{(i)}$ and 
$x_{t_i}^{(i)}$ are the endpoints 
of $S_i$, and (3) such that 
at least $t_i/4$ sites on $\pi_i$ (determined
as a function of $\pi_i$) are closed. 

Form a path $\pi$ by connecting together 
all intervals in $L\setminus S$ and all paths $\pi_i$. 
As $|L\setminus S|\le \sum_it_i$, the proportion 
of closed sites on $\pi$ is at least $1/8$. Trivially, the 
length $t$ of $\pi$ is at least $\ell$. 
It follows that 
$$
\prob{F_\ell^c}\le \prob{\pi\text{ exists}}\le \sum_{t\ge \ell}8^t(1-p)^{t/8},
$$
which establishes Step 2.

\noindent{\bf Step 3.} Conclusion of the proof, by the pigeonhole principle. 

Let $L'=\{(i,\ell+1):1\le i\le \ell\}$, and $F_\ell'$ the event that strictly more than $\ell/2$ sites on $L'$ 
are connected to $G_0$ through oriented percolation paths.
If $p$ is close enough to $1$, then by symmetry and Step 2, 
$$
\prob{F_\ell\cap F_\ell'}\ge 1-\exp(-\gamma\ell),
$$
for some constant $\gamma>0$ 
(not depending on $\ell$). 
Suppose that $F_\ell\cap F_\ell'$ occurs. Then, by Step 1 and the pigeonhole 
principle, there exists an 
$i\in [1,\ell]$, so that 
$(1,i)\in L$ and  $(i,\ell+1)\in L'$ are 
eventually occupied, in which case $(1,\ell+1)$ becomes occupied
if open. It follows that 
$$
\prob{(1,\ell+1)\text{ is open but never occupied}}
\le \prob{(F_\ell\cap F_\ell')^c}\le \exp(-\gamma\ell).
$$
Therefore, 
\begin{equation*}
\begin{aligned}
&\prob{\text{there is an open edge of length at least 
$C\log n$ that is never occupied}}\\&\le 
n^2\exp(-\gamma C\log n)\ll1
\end{aligned}
\end{equation*}
for any $C>2/\gamma$.
\end{proof}

The final task of this section 
is to address saturation for Catalan percolation. 

\begin{proof}[Proof of Theorem~\ref{tc3} (3)]
For $1\le i\le n-3$, 
let $Z_i$  be the indicator of the event that 
the edge $i\to(i+3)$ is open but never occupied (i.e., $i\to (i+2)$
and $(i+1)\to (i+3)$ are both closed). 
The random variable $N=\sum_i Z_i$ has $\E N=(n-3)(1-p)^2p$. 
Since $p=1-\alpha n^{-1/2}$, it follows that   
$\E N\approx n\cdot (\alpha n^{-1/2})^2\cdot 1=\alpha^2$. 
Furthermore, $N$ converges in distribution 
to a Poisson($\alpha^2$) random variable 
by an application of the Chen--Stein method \cite{BHJ}. 
Indeed, $Z_i$ and $Z_j$ are independent unless $|i-j|\le 1$,  therefore the total variation distance between 
(the distribution of) $N$ and
Poisson($\E N$) is bounded above by
\begin{equation*}
\begin{aligned}
&\sum_i\left[(\E Z_i)^2+\sum_{j:|i-j|=1} (\E Z_i\E Z_j+\E(Z_iZ_j))\right]\\
&\le n\left[(1-p)^4+2((1-p)^4+(1-p)^3)\right]
=\cO(n^{-1/2}).
\end{aligned}
\end{equation*}
Therefore,
\begin{equation}\label{eq-super-0}
\begin{aligned}
&\limsup_n\prob{\text{all open oriented edges become occupied}}
\\&\le 
\limsup_n\prob{N=0}=
\exp(-\alpha^2).
\end{aligned}\end{equation}
Now let 
$H_\ell$ be the event that 
$\ell$ is the minimal length of an unoccupied open edge. 
Note that if $N=0$ then all open edges of length 3 become occupied. 
Therefore 
\begin{equation}\label{eq-super-1}
\prob{\text{all open oriented edges become occupied}}
=
\prob{N=0}-\sum_{\ell\ge4}\prob{H_\ell}. 
\end{equation}
Note that, on the event  $H_\ell$, there is an 
edge $(i,i+\ell)$ so that, for all $1\le j< \ell$, 
either $i\to(i+j)$ or 
$(i+j)\to(i+\ell)$ is closed. It follows that 
\begin{equation}\label{eq-super-2}
\begin{aligned}
\prob{H_\ell}\le n\cdot 2^{\ell-1}(1-p)^{\ell-1}= (2\alpha)^{\ell-1}n^{1-(\ell-1)/2}.
\end{aligned}
\end{equation}
By \eqref{eq-super-1} and \eqref{eq-super-2}, 
\begin{equation}\label{eq-super-3}
\begin{aligned}
&\prob{\text{all open oriented edges become occupied}}\\
&\ge \prob{N=0}- \frac{(2\alpha)^3}{\sqrt{n}}\sum_{\ell\ge0}(2\alpha/\sqrt{n})^\ell\\
&= \exp(-\alpha^2)- \cO(n^{-1/2}).
\end{aligned}
\end{equation}
Putting the bounds \eqref{eq-super-0} and \eqref{eq-super-3} 
together 
completes the proof. 
\end{proof}

\section{Intermediate regime for linear initial graphs}\label{sec-tricky}

For an edge $i\to j$, we say that another edge $x\to y$ is {\it below} $i\to j$ if 
$x,y$ are between $i,j$. Similarly, we say that $x\to y$ is {\it above} $i\to j$ if
one of $x,y$ are on either side of $i,j$. 

In the Catalan percolation process studied above, 
where all edges are oriented in the same direction, 
an edge $i\to j$ can only become occupied due to other
edges below $i\to j$ becoming (or being initially) occupied. 
On the other hand, if leftward and rightward edges are present (initially 
occupied or open), 
then there are many ways in which they can interact, 
leading to the eventual occupation of various edges. 
See Fig.~\ref{fig-LRinteract}.

\begin{figure}[h!]
\centering
\includegraphics{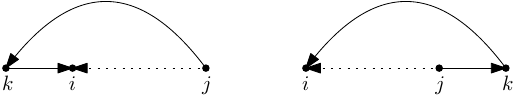}
\caption{A leftward open edge $i\lto j$
becomes occupied due to  
occupied edges $i\lto k\lto j$ of opposite orientations, 
for some $k\notin[i,j]$.}
\label{fig-LRinteract}
\end{figure}

\subsection{The tilde process}
\label{subsec-left-subc}

\newcommand{\tE}{\widetilde E}
\newcommand{\tEo}{\widetilde E_{\text{\rm open}}}
\newcommand{\tI}{\widetilde I_e}

It appears challenging to accurately control  
the interactions between leftward and rightward edges in any regime
in between that of Theorem~\ref{tc2} (1) and (3). 
We present a modest result Theorem~\ref{tc2} (2)
stating that,  when $\pl<c/{\sqrt{\log n}}$ and 
$\pr>A$, for small enough $c$ and large enough $A$, a.a.s.\ 
all open rightward edges longer than $\alpha\log n$ are eventually occupied, 
however, no such leftward edges are ever occupied. 
The statement about rightward edges follows by Lemma~\ref{catalan-sup} above.
To prove the other statement, we show that, even if {\it all} rightward edges
were to become occupied, a.a.s.\ no long leftward edges become 
occupied.  

More formally, we consider a modified {\it tilde process}
$\tE_t$, which describes the occupation of leftward edges in time, 
when {\it all} rightward
edges are assumed to be initially occupied.  
The set $\tE_0$ of initially occupied leftward 
edges is given by $1\lto 2\lto \cdots\lto n$.
The set $\tEo$ of open leftward edges 
is obtained by opening leftward edges (of 
length at least $2$) independently with probability $\pl=p$.
Given $\tE_t$, an edge $i\lto j\in \tE_{t+1}$, provided that  $i\el j\in \tEo$ 
and for some $k\in [n]$ we have that: 
\begin{itemize}
\item $i\el k\el j \in \tE_{t}$ and $i<k<j$; or
\item $k\el j\in \tE_{t}$ and $k<i$; or 
\item $i\el k\in \tE_{t}$ and $k>j$. 
\end{itemize}

In other words, in each step of the tilde process, either   an open leftward
edge becomes occupied due to a usual transitive step, or else,   some open leftward
edge becomes occupied which is below, and shares an endpoint with, 
a previously occupied leftward edge.

To show that these dynamics are 
subcritical for $p<c/\sqrt{\log n}$, when $c$ is small, 
we translate some of the ideas and definitions 
from Section~\ref{sec-general-subcr}.

We call an interval $I\subset [n]$ \emph{good}
if either
\begin{itemize}
\item $|I|= 2$; or 
 \item $|I|\ge 3$ and, for every $\{i,i+1\}\subset I$, there exists 
a $j\in I$ so that either (1) $j<i$ and 
$i\to j\lto i+1\in \tEo\cup \tE_0$, 
or else, (2) $j>i+1$ and 
$i\lto j\to i+1\in \tEo\cup \tE_0$.
\end{itemize}

When such a $j$ exits for $i\in I$, we say that $i$ is the base of a {\it tilde horn}
in $I$. 

Assume that an edge $e=i_1\el i_2\in \tE_\infty$. 
Associated with $e$, we let $\tI$ denote an interval $I$
of minimal cardinality such that graphs on $I$ induced by edges in 
$\tE_0\cup \tEo$ make $e$ occupied (by the tilde process dynamics). 
Note that $[i_1,i_2]\subset \tI$. 

The next lemma is an analogue of Lemma~\ref{Ie-lemma}.

\begin{lemma} \label{Ie-good} 
For any $e\in \tE_\infty$, 
the interval $\tI$ is good. That is, either $e\in \tE_0$, or else, 
for each $\{i,i+1\}\subset \widetilde {I_e}$, $i$ is the base of a 
tilde horn in $\tI$. 
\end{lemma}

\begin{proof}
For any $\{i_0, i_0+1\}\subset \widetilde I_e$, an open edge over the initially 
occupied edge $i_0\lto (i_0+1)$ 
must become 
occupied, or else the interval could be shortened.  
The first time $t_0\ge 1$ such an edge $i\el j$ becomes occupied,
it follows by the minimality of $t_0$ that $i\lto k\lto j\in E_{t_0-1}$
for some $i<k<j$. That is, $i\lto j$ becomes occupied by a usual transitive step. 
Moreover, again by the minimality of $t_0$, either 
(1) $i=i_0$ and $k=i_0+1$, 
or else, (2) $k=i_0$ and $j=i_0+1$.  
\end{proof}

We also need counterpart of Lemmas~\ref{AL-property} and \ref{L_BK}. 
We omit the proof, since they are almost identical.  

\begin{lemma} \label{AL-tricky}
Assume that $e_0\in \tE_\infty$ has length $\ell$. Then,
for every integer $k\in [1, \ell]$,  
there exists an edge $e$ with $|\tI|\in [k+1,2k]$. 
\end{lemma}

\begin{lemma}\label{BK-tricky}
Suppose that $K\subset V$ is such that all $v\in K$ are bases of tilde horns in $K$. 
Then there is a set $K_0\subset K$ of size at least $|K|/(9D)$
so that tilde horns (in $K$) for each $v\in K_0$ can be chosen so that their
edge-sets are pairwise disjoint. 
\end{lemma}

\begin{proof}[Proof of Theorem~\ref{tc2}~(2)]
As already mentioned, the statement 
for rightward edges follows by Lemma~\ref{catalan-sup}.
The statement for leftward edges
can be proved along the same lines as Theorem~\ref{tg-subcr}, but using Lemmas~\ref{Ie-good}--\ref{BK-tricky} instead of Lemmas~\ref{Ie-lemma}--\ref{L_BK}.
\end{proof}

\section{Supercritical regime for bounded-degree initial graphs}\label{sec-supercr}

\newcommand{\abun}{\text{\tt abundance}}
\newcommand{\horns}{\text{\tt horns}}
\newcommand{\connect}{\text{\tt connect}}
\newcommand{\sat}{\text{\tt first\_sat}}

Finally, we prove the supercritical result Theorem~\ref{tc1} (2)
for bounded-degree initial graphs $G_0$. 
Recall that this result implies Theorem~\ref{tc2} (3). 
Before turning to the proof, we state a few
preliminary observations, and briefly discuss some of the main parts
of our strategy.  

First, note that, we can assume that $G_0$ is an unoriented tree
(i.e., replace $G_0$ it by a spanning subtree if necessary). 
Then using a result from \cite{BCDS}, which  
follows from Lemma 6.1 in \cite{GS}, we obtain a large number 
of edge-disjoint subtrees of   
$G_0$ of some (suitably chosen) size. 

\begin{lemma}\label{tree-cover} For any 
tree with $n$ vertices and integer $L\in [1,n-1]$ 
there exist $\lceil (n-1)/(2L^2)\rceil$ subtrees such that 
(1) each subtree has $L$ edges, and (2) any two subtrees have at 
most 1 vertex in common. 
\end{lemma}

Next, we prove the following lemma, which we will use in 
describing the spread of occupied edges via nucleation
in the supercritical regime. Note that, once again, horns 
are playing a crucial role in our arguments. 

\begin{lemma}\label{L_nuc}
Assume that some subtree $T\subset G_0$ is internally saturated. 
Suppose that for every neighboring (in $G_0$) vertices 
$v,v'$ which are not both in $T$, 
(1) there is a $y\in T$ so that edges  
$v\lto y\to v'\in \Eo\cup E_0$ (oriented away from $T$), 
and (2) the set
$U_v^\to$ of endpoints $u\in T$ of edges $u\to v$ is 
strongly connected
by edges in $\Eo\cup E_0$. 
Then all edges from $T$ to $G_0\setminus T$ are eventually occupied. 
Likewise, a symmetric statement also holds in the reverse orientation
(i.e., towards $T$). 
\end{lemma}

\begin{proof}
This follows by a straightforward induction on the distance 
(in $G_0$) of a vertex $v\notin T$ to $T$. 
In (1),
we choose $v'$ to be the neighbor  
(in the unoriented  tree $G_0$) of $v$ that is closest to $T$. 
Then, by the inductive hypothesis, $y\to v'$ is eventually occupied. 
Therefore, since $v'\leftrightarrow v\in E_0$, it follows that $y\to v$ is eventually 
occupied. Next, by (2), there is a collection of oriented
paths such that all (a) end at $y$, (b) visit only vertices in $U_v^\to$, 
and (c) together visit all points in 
$U_v^\to$. Therefore, starting with the eventually occupied edge $y\to v$, it follows by another induction (on the distance to $y$ along such oriented paths 
ending at $y$) that 
all edges from $U_v^\to$ to $v$ are eventually occupied. 
Informally, we can backtrack (started from $y$) along such a path to $y$ until we eventually reach any given $u\in U_v^\to$, occupying edges from this path 
to $v$ along the way. 
\end{proof}

Therefore, supposing that one of the subtrees $T\subset G_0$ (of size $L$, to be 
determined below) given by 
Lemma~\ref{tree-cover} is internally saturated, 
all other open edges in $\Eo$ become occupied by 
Lemma~\ref{L_nuc}, provided that $L$ is large enough so that,  
a.a.s.\ for all $x,y\notin T$,  there are edges $x\to u\to y$ for some $u\in T$. 
Showing that at least one such subtree is internally saturated
follows similarly, however, on this smaller
scale slightly more delicate arguments are required. 

In order to apply Lemma~\ref{L_nuc}, we will require the following standard result about the connectivity 
of oriented \ER random graphs, the proof of which we only briefly sketch.

\begin{lemma} \label{oriented-ER-connectivity}
Assume $G$ is an oriented 
\ER random graph on $n$ points with edge probability $p$. If $p=c\log n/n$ with $c>1$,
then 
$$
\prob{G\text{ is  not strongly connected}}=\cO(n^{1-c}).
$$
If $p=n^{-\alpha}$, for some $\alpha<1$, then 
$$
\prob{G\text{ is  not strongly connected}}\le \exp(-n^{1-\alpha}/2).
$$
\end{lemma}

\begin{proof}
If $G$ is not strongly connected, then there exists a nonempty 
set $A$ of $k\le n/2$ 
points so that there are no outward connections, or no inward connections, from 
$A$ to
$A^c$. Therefore (using the bounds ${n\choose k}\le (ne/k)^k$
and $(1-x)\le e^{-x}$), 
\begin{equation*}
\begin{aligned}
&\prob{G\text{ is not strongly connected}}\\
&\le 2\sum_{k=1}^{\lfloor n/2\rfloor}\binom{n}{k}(1-p)^{k(n-k)}\\
&\le 2\sum_{k=1}^{\lfloor n/2\rfloor}\exp[-k(pn+\log k-pk-\log n-1)].
\end{aligned}
\end{equation*} 
The desired inequalities then follow by dividing the above sum 
into two sums over $k\le p^{-1/2}$ and $k>p^{-1/2}$.
\end{proof}

Finally, we note that 
in our context, an event $\mathcal E$
(i.e., a subset of the sample space $\Omega$ of all possible configurations $\omega$ of open and closed edges) 
is {\it increasing}  
if $\omega \in E$ implies $\omega_+ \in E$ whenever 
$\omega_+$ contains
    all open edges in $\omega$. 
In other words, an event is increasing if it 
cannot be destroyed by the addition of open edges. 
Note that the event 
$\{V\text{ is saturated}\}$ is not increasing. To deal 
with this nuisance, we say that a set $E$ 
of oriented edges
between vertices in $V$  
is {\it abundant\/} if
for every $i,j\in V$, there exists 
$k\in V$ so that $i\er k\er j\in E$. We record the 
following simple observations. 

\begin{lemma}\label{ab-sat}
The event $\{E_\infty\text{ is abundant}\}$ is increasing  and 
$$
\{E_\infty\text{ is abundant}\}\subset \{V\text{ is saturated}\}. 
$$
\end{lemma}

We now prove the following result, 
which immediately implies 
Theorem~\ref{tc1} (2).

\begin{theorem} \label{supercr-thm} 
Assume that $G_0=(V,E_0)$ 
is an unoriented  connected graph. Suppose  
that open 
(oriented) edges are chosen independently (from amongst 
those not in $E_0$) with probability $\po\ge C\log\log n/\sqrt{\log n}$, for some $C>4$.
Then, with high probability,  
$V$ is saturated.
\end{theorem}

\begin{proof} 
We divide the proof into several steps.  
In Step 1, we select edge-disjoint subtrees of $G_0$ of suitable sizes. 
Steps 2 and 3 show that if any of these trees are saturated, then a.a.s.\ 
so is $V$. By similar, but more delicate reasoning, we show 
in Step 4
that indeed 
a.a.s.\ at least
one such tree is internally saturated. The final Step 5 
extends these results to larger $p$.

\noindent{\bf Step 1.}
Recall that, for simplicity we may assume,  without loss of generality,
that $G_0$ is edge-minimal, that is, a
spanning tree. 
Fix $C>8$ and  put 
\[
k=\left\lceil \frac{\log n}{2\log\log n}\right\rceil, 
\quad 
p=\sqrt{\frac {C\log k}k}.
\]
By Lemma~\ref{tree-cover}, we fix subtrees $T_m$, $m=1,\ldots, \lceil n/(4k^{6})\rceil$, 
of size $k^3$,  
no two of which share more than a single vertex. 
Generate the configuration of open 
edges $\Eo$ with $\po=p$.

\noindent{\bf Step 2.} We claim that a.a.s.\ all subtrees $T_m$ have the following properties: 
\begin{enumerate}
\item For all $j_1,j_2\in[n]$ there are $i_1,i_2,i_3\in T_m$ such that 
all edges 
$j_1\lto i_1\to j_2$, $j_1\to i_2\lto j_2$ and 
$j_1\to i_3\to j_2$  
are in $\Eo\cup E_0$. 

In particular, for every $j_1\lr j_2\in E_0$, 
the edges $j_1\lto i_1\to j_2$, $j_1\to i_2\lto j_2$
give horns $(j_1,j_2,i_1)$ and $(j_1,j_2,i_2)$, oriented towards and away from $j_1$ with their tips in $i_1,i_2\in T_m$.

\item 
For all $j\notin T_m$, the sets 
\[
\begin{aligned}
&U_j^\to=\{i\in T_m:i\to j\in \Eo\cup E_0\},\\ 
&U_j^\lto=\{i\in T_m:i\lto j\in \Eo\cup E_0\}
\end{aligned}
\]
are strongly connected
by edges in $\Eo\cup E_0$. 
\end{enumerate}
 
To see this, we first claim  
that, for any given $T_m$ and $j_1,j_2\in[n]$, the probability that 
property (1) fails is at most, for all large $n$,  
\[
3(1-p^2)^{k^3-2}\le 4\exp(-p^2k^3)\le 4\exp(-k^2). 
\] 
This follows by a union bound: 
For a given $j_1,j_2$ there are at least $k^3-2$ vertices
$i\in T_m\setminus\{j_1,j_2\}$. If (1) fails then  
either (a)  
for all such $i$ at least one of the edges $j_1\lto i\to j_2$ is closed,
(b) for all such $i$ at least one of edges $j_1\to i\lto j_2$ is closed,
or (c) for all such $i$ at least one of edges $j_1\to i\to j_2$ is closed. 
For fixed $j_1,j_2$ any one of these events holds with 
probability at most $(1-p^2)^{k^3-2}$ by independence. 

By Lemma~\ref{oriented-ER-connectivity} above 
and standard Binomial tail bounds (e.g., Lemma 2.8 in \cite{GS}), for any given $T_m$
and $j\notin T_m$, the probability that any given 
$U_j^\to$ or $U_j^\lto$  
is not strongly connected
is at most, for all large $n$,  
\begin{multline*}
\prob{{\rm Bin}(k^3,p)\le pk^3/2}+\exp(-k^2/2)\\
\le 
\exp(-pk^3/7)+\exp(-k^2/2)\le 2\exp(-k^2/2). 
\end{multline*}
Hence, for all large $n$, all trees $T_m$ have properties (1) and (2) with probability at least
\[
1-\frac{n}{2k^6}[4n^2\exp(-k^2)+4n\exp(-k^2/2)]
\ge1-n^3\exp(-k^2/2)=1-o(1).
\]

\noindent{\bf Step 3.} Convert all open edges between vertices of $T_m$ to occupied. 
We claim that properties (1) and (2)
for $T_m$ imply that all other open edges (not between vertices in $T_m$) 
are eventually occupied. 

Indeed, using the horns provided by (1), the strong connectivity in (2) and Lemma~\ref{L_nuc}, all open edges with exactly one 
endpoint in $T_m$ are eventually occupied. As discussed 
below the proof of Lemma~\ref{L_nuc}, all other edges between $x,y\notin T_m$
are then occupied, using the edges $x\to i\to y\in\Eo\cup E_0$
for some $i\in T$, provided by (1).

\noindent{\bf Step 4.} A.a.s., some $T_m$ is saturated. 

We show that any given subtree $T_m$ is saturated with probability at least
$(2\sqrt{n})^{-1}$. Given this, recalling that any two subtrees
share at most 1 vertex, it follows that some $T_m$ is saturated with probability 
at least 
\[
1-(1-(2\sqrt{n})^{-1})^{n/(4k^{6})}\ge 1-\exp(-\sqrt n/(8k^{6}))=1-o(1). 
\]

Since the $T_m$ are of the same size, it suffices to consider the case $T_1$. 
Moreover, for notational convenience, let us assume that 
$T_1=[1,k^3]$ and that for all $j\le k^3$ the vertices in $[1,j]$ form
a subtree of $T_1$. 

\noindent{\bf Step 4a.} For all large $n$, with probability at least 
$n^{-1/2}/\log n$ 
all edges $1\lrto i\in[2,k]$,  are in $\Eo\cup E_0$ and hence 
$[1,k]$ is saturated. 

Indeed, for large enough $n$,  
all such edges are in $\Eo\cup E_0$ with probability at least 
\[
p^{2k}\ge(\log n)^{-k}
\ge n^{-1/2}/\log n. 
\]
By induction, all edges $1\lrto i$ become occupied,
and using these edges all other open edges can be occupied:
if $i\to j$ is open, then it becomes occupied due to the occupied edges 
$i\to 1\to j$. 

\noindent{\bf Step 4b.} A.a.s., for any $j_1,j_2\in[k+1,k^3]$ there are $i_1,i_2,i_3\in[1,k]$
such that 
all edges 
$j_1\lto i_1\to j_2$, $j_1\to i_2\lto j_2$ and 
$j_1\to i_3\to j_2$  
are in $\Eo\cup E_0$. 

These edges play a similar role as those in Step 2 above. \
Moreover, the existence of such edges is proved similarly. 
For fixed $j_1,j_2$, a requiste $i_1$, say, will 
fail to exist with probability $(1-p^2)^{k}$ by independence.
Therefore, noting that $1-x\le e^{-x}$ and $kp^2=C\log k$, 
such edges are not open with probability at most (recall $C>8$)  
\[
3k^{6}(1-p^2)^{k}\le 6k^{6-C}\ll1.
\]

Next, for $j\in[k+1,k^3]$, we consider sets 
$V_j^\to,V_j^\lto$ analogous to the sets $U_j^\to,U_j^\lto$\
considered in  Step 2 above. However, in the present setting
(where the subtree on $[1,k]$ is much smaller than $T_m$ of size $k^3$),  
strong connectivity no longer follows by 
a simple union bound. 

\noindent{\bf Step 4c.}
For $j\in [k+1, k^3]$, we claim that the sets  
\begin{equation*}
\begin{aligned}
&V_j^\to=\{i\in [1,k]: i\to j\in \Eo\cup E_0\}, \\ 
&V_j^\lto=\{i\in [1,k]: i\lto j\in \Eo\cup E_0\}
\end{aligned}
\end{equation*}
are a.a.s.\ 
strongly connected by edges in $\Eo\cup E_0$.

Let $F_j^\to$ (resp.~$F_j^\to$) be the event that 
$V_j^\lto$ (resp.~$V_j^\to$) 
is strongly connected by edges in $E_0\cup \Eo$.
Let 
$$B=\bigcap_{j\in [k+1, k^3]}\left(F_j^\lto\cap F_j^\to\right).$$
The crucial step is the following correlation inequality
\begin{equation}\label{super-thm-eq1}
\prob{B}\ge \prod_{j\in [k+1, k^3]}\prob{F_j^\lto}\prob{F_j^\to}.
\end{equation}
To prove (\ref{super-thm-eq1}), 
let $\cA$ be the 
set of all possible choices of $V_j^\lto$, $V_j^\to$, 
that is, the set that contains all ordered selections
of $2(k^3-k)$ subsets of $[1,k]$:
$$
\cA=\{(A_j^\lto, A_j^\to:j=k+1,\ldots k^3): A_j^\lto,A_j^\to\subset [1,k]\text{ for all }j\}.
$$ 
Observe that for any vector $(A_j^\lto,A_j^\to)_j$ of such  
(deterministic) subsets, the events $\{V_j^\lto=A_j^\lto\}$, $\{V_j^\to=A_j^\to\}$, $j\in[k+1,k^3]$, 
are independent. Therefore, with  indices $j$ and $j'$ running 
over $[k+1,k^3]$, 
\begin{equation*}
\begin{aligned}
\prob{B}&=
\sum_{(A_j^\lto,A_j^\to)\in \cA}
\prob {\cap_j (F_j^\lto\cap F_j^\to)\;\bigcap\; \cap_{j'}\{V_{j'}^\lto=A_{j'}^\lto,V_{j'}^\to=A_{j'}^\to\}}
\\
&=\sum_{(A_j^\lto,A_j^\to)\in \cA}\prob{\cap_j (F_j^\lto\cap F_j^\to)\mid 
\cap_{j'}\{V_{j'}^\lto=A_{j'}^\lto,V_{j'}^\to=A_{j'}^\to\}}\\
&\quad\quad\times
\prob{\cap_{j'}\{V_{j'}^\lto=A_{j'}^\lto,V_{j'}^\to=A_{j'}^\to\}}
\\
&\ge \sum_{(A_j^\lto,A_j^\to)\in \cA}
\prod_{j}
\prob{ F_j^\lto\mid \cap_{j'}\{V_{j'}^\lto=A_{j'}^\lto,V_{j'}^\to=A_{j'}^\to\}}\\
&\quad\quad\times
\prob{ F_j^\to\mid \cap_{j'}\{V_{j'}^\lto=A_{j'}^\lto,V_{j'}^\to=A_{j'}^\to\}}
\prob{\cap_{j'}\{V_{j'}^\lto=A_{j'}^\lto,V_{j'}^\to=A_{j'}^\to\}},\\
\end{aligned}
\end{equation*}
by the Fortuin--Kasteleyn--Ginibre inequality \cite{FKG}. Hence 
\begin{equation*}
\begin{aligned}
\prob{B}
&\ge  \sum_{(A_j^\lto,A_j^\to)\in \cA}\prod_{j}\prob{ F_j^\lto\mid V_{j}^\lto=A_{j}^\lto}
\prob{ F_j^\to\mid V_{j}^\to=A_{j}^\to}\\
&\qquad\quad\quad\times
\prod_{j'}\prob{V_{j'}^\lto=A_{j'}^\lto}\prob{V_{j'}^\to=A_{j'}^\to}\\
&= \sum_{(A_j^\lto,A_j^\to)\in \cA}\prod_{j}
\prob{ F_j^\lto\mid V_{j}^\lto=A_{j}^\lto}
\prob{V_{j}^\lto=A_{j}^\lto}\\
&\qquad\quad\quad\times
\prob{ F_j^\to\mid V_{j}^\to=A_{j}^\to}
\prob{V_{j}^\to=A_{j}^\to}\\
&= \sum_{(A_j^\lto,A_j^\to)\in \cA}\prod_{j}\prob{ F_j^\lto\cap\{ V_{j}^\lto=A_{j}^\lto\}}
\prob{ F_j^\to\cap\{V_{j}^\to=A_{j}^\to\}}
\\
&= \prod_j\left(\sum_{A_j^\lto\subset [1,k]}\prob{ F_j^\lto\cap\{ V_{j}^\lto=A_{j}^\lto\}}\right)
\left(\sum_{A_j^\to\subset [1,k]}\prob{ F_j^\to\cap\{ V_{j}^\to=A_{j}^\to\}}\right)
\\
&= \prod_j\prob{ F_j^\lto} \prob{ F_j^\to}.
\end{aligned}
\end{equation*}
Moreover, by Lemma~\ref{oriented-ER-connectivity} above 
and standard tail bounds (e.g., Lemma 2.8 in \cite{GS}),
for large $k$, 
$$
\prob{(F_j^\lto)^c}\le \prob{|V_j|\le pk/2}+ k^{1-C/2}
\le \exp(-pk/7)+ k^{1-C/2}\le 2k^{1-C/2},
$$
and a similar bounds holds for $F_j^\to$. 
It follows (by Bernoulli's inequality) that, for large $k$,  
$$
\prob{B}\ge \left(1-2k^{1-C/2}\right)^{2k^3}
\ge 1-4k^{4-C/2}=1-o(1),
$$
since $C>8$.

\noindent{\bf Step 4d.} For all large $n$, $T_1$ is saturated with 
probability at least $(2\sqrt{n})^{-1}$. 

Note that, for all large $n$, 
the claims in the previous three steps all hold with probability 
at least $(2\sqrt{n})^{-1}$. Hence it remains to show that 
they together imply that $T_1$ is saturated. 
However, this follows by a similar argument as 
was used in Step 3 above, but using Steps 4a--c instead of 
Step 2. 

Altogether, by Step 4, a.a.s.\ some subtree $T_m$ is saturated, and thus by 
Steps 2 and 3, a.a.s.\ $V$ is saturated. 

\noindent{\bf Step 5.}
Finally, we extend our results from the case $\po=p$ to larger $\po$.  
This follows by the simple observation that, for all large $n$, 
\[
\prob{\text{$\Eo$ is not abundant}}\le 
2n^2(1-p^2)^{n-2}\le 3n^2e^{-p^2n}\le 3n^2e^{-n/\log n}\ll1.
\]
Therefore, for $\po=p$, a.a.s.\ $E_\infty$ is abundant since we have shown
that a.a.s.\ $V$ is saturated (i.e., $\Eo\subset E_\infty$). 
Hence, 
by Lemma~\ref{ab-sat}, a.a.s. $E_\infty$ is abundant for $\po\ge p$, 
and so also, a.a.s.\ 
$V$ is saturated for $\po\ge p$.
\end{proof}

\subsection{$R$-unoriented initial graphs}

We can relax the assumption that $G_0$ is unoriented, 
but we emphasize that strong connectivity of $G_0$ is not enough for 
Theorem~\ref{supercr-thm} to hold in the same form (see the discussion on Open Problem~\ref{con-oriented}). 
We only 
provide the following mild generalization, whose 
proof is omitted as it is a minor adaptation of the proof
of Theorem~\ref{supercr-thm}. 
Informally, we start with an unoriented tree $T$ and replace 
every vertex of $T$ with a graph of bounded size that
is strongly connected, so that between $T$-neighboring sets we have edges in both directions.
To be more precise, for an integer $R\ge 1$, 
we say that $G_0$ is {\it $R$-unoriented\/}  if there exists an unoriented  
tree $T$ on a vertex set $V'$, together with 
a map $\phi:V\to V'$, such that: 
(1) $|\phi^{-1}(y)|\le R$ and $\phi^{-1}(y)$ is strongly connected 
for all $y\in V'$; and (2) if $y_1,y_2$ are neighbors in 
$T$, then there are $x_1\in \phi^{-1}(y_1)$ and 
$x_2\in \phi^{-1}(y_2)$, such that $x_1\to x_2\in E_0$.

Note that $1$-unoriented graphs are exactly those with an unoriented 
spanning tree. 
For an example with $R=2$, take $V =[2n]$ and assume
$1\une2,3\une4,5\une6, \ldots, (2n-1)\une(2n)$ are strongly connected pairs,
and add connections $1\to3$, $2\lto4$, $3\to5$, $4\lto6$, etc. Here $T$ is 
a linear graph on $[n]$. 

\begin{theorem} \label{supercr-thm-R} 
If $G_0$ is an $R$-unoriented  connected graph, 
we have that 
$\po\ge C\log\log n/\sqrt{\log n}$,
and $C\ge C_0(R)$, then 
$E_\infty$ is a.a.s.~saturated.
\end{theorem}

\section{Open problems}\label{sec-open}

For clarity, each unresolved issue is presented in what we view as the simplest 
context, although most can be studied in much greater 
generality. We begin with a conjecture about a sharp transition 
in Catalan percolation.

\begin{conjecture}\label{con-cat}
There exists a critical probability 
$p_c^{\text{\rm Cat}}\in (0,1)$
so that for $p < p_c^{\rm Cat}$ (resp.\ $p>p_c^{\rm Cat}$) 
there exists a constant $C=C(p)$ so that a.a.s.\ 
$E_\infty$ in the Catalan percolation process contains no edge 
(resp.\ contains all open edges)
of length at least $C\log n$.
\end{conjecture}

On the other hand, in the case of $G_0=L_n$, when both $\pr>0$ and $\pl>0$, the interaction
between leftward and rightward edges is a challenge.   

\begin{openproblem}
In the setting of Theorem~\ref{tc2}, is it true that 
when $\pl<c\frac 1{\sqrt{\log n}}$ and 
$\pr<a$, a.a.s.\ $E_\infty$ contains no 
edges longer than $\alpha\log n$?
\end{openproblem}

For the statements of our remaining 
open problems, we define 
$$
p_c=\inf\{p:\prob{\text{$V$ is saturated}}\ge 1/2\text{ for all }\po\ge p\}.
$$
Perhaps the most pressing remaining question 
is the correct power of  $\log\log n$ for the 
transition in Theorem~\ref{tc1}. We suspect 
neither bound in that theorem is sharp, as the existence of a giant component,  
rather than connectivity of edge endpoints (as used in the 
proof of Theorem~\ref{supercr-thm}) should suffice. 
We assume the unoriented setting in our next four open problems 
(i.e., 
that $G_0$ and $\Go$ are both unoriented) and that $\Go$ is 
the \ER graph with 
probability  $\po$ of open edges. 

\begin{conjecture}
\label{con-loglog} Assume that $G_0$ is the linear 
graph on $[n]$. 
Then we have that 
$p_c=\Theta(\sqrt{\log\log n/\log n})$.  
\end{conjecture}

Graphs of bounded diameter are, in a way, at the opposite 
extreme from graphs of bounded degree. As in the case of 
bootstrap percolation, the scaling of 
the critical probability should change dramatically. 

\begin{conjecture}
\label{con-hamming} Assume that $V=[n]^d$ and that $G_0$ is
the Cartesian product of $d$ complete graphs on $[n]$, i.e., 
the $d$-dimensional Hamming graph. For $d\ge 3$, there 
exists a power $\gamma=\gamma(d)\in (0,\infty)$ so that, for every 
$\epsilon>0$ and large enough $n$,  $p_c$ is between 
$n^{-\gamma-\epsilon}$ and $n^{-\gamma+\epsilon}$.
\end{conjecture}

Observe that the above conjecture does not hold for 
$d=2$ (or for any other $G_0$ with diameter $2$), 
in which case all open edges get occupied at time 
$1$, regardless of $\po$. 
We suspect that  in the
setting of Conjecture~\ref{con-hamming} the threshold 
$p_c$ is not sharp, in the sense of \cite{FK}. 
Indeed, computer simulations
suggest that saturation fails close to criticality due to rare
open edges with a protective arrangement of nearby closed edges, and that
the number of such protective local configurations approaches a Poisson
distribution with a parameter that depends continuously on the constant
$a$ if $p=a p_c$. By contrast, we conjecture that $p_c$ is sharp in
Conjecture~\ref{con-loglog}. The 
methods of \cite{FK} (or subsequent work) do not apply in any of these cases, 
as our random objects (edges) 
do not play symmetric roles.  

Perhaps the most interesting intermediate case is the 
hypercube, for which we have no guess about the 
size of $p_c$.

\begin{openproblem}
\label{con-hypercube} Assume 
that $G_0$ is the 
hypercube on $\{0,1\}^n$. What is the asymptotic behavior of
$p_c$?
\end{openproblem}

Another natural graph with unbounded degree is the random graph. 

\begin{openproblem}
Assume  $G_0$ is an \ER graph with edge probability $p_{\rm initial}$.
Estimate the probability of saturation, 
in terms of $p_{\rm initial}$ and $\po$.
\end{openproblem}

More complex
edge addition dynamics can be considered in polluted environments.
Following the lead of \cite{Bol, BBM}, we consider {\it $K_d$-percolation\/}, whereby we iteratively complete
all copies of $K_d$ missing a single edge, where $K_d$ is
the complete graph on $d$ points. We assume that 
$G_0$ the graph on $[n]$ 
with edges $i\une j$, for all $|i-j|\le d-2$. Note that this is the simplest 
initialization that results in saturation when $\po=1$.  Simulations suggest (see
left panel 
of Fig.~\ref{fig-op})  that nucleation occurs 
for all $d\ge 3$. The unpolluted ($\po=1$) version 
of this process
is analyzed in, e.g., \cite{BBM,BBMR12,BPRS17,GKP17,AK, Kol, BK20}.

\begin{conjecture}
\label{con-K4} Consider the $K_d$-percolation dynamics, with  
$G_0$ as above. Then there exists 
some power $\gamma=\gamma(d)>0$ so that we have that 
$p_c=\Theta[(\log\log n)^{\gamma}(\log n)^{-1/(d-1)}]$.
\end{conjecture}

\begin{figure}[ht!]
\centering
\includegraphics[trim=0cm 0cm 0cm 0cm, clip, width=.3\textwidth]{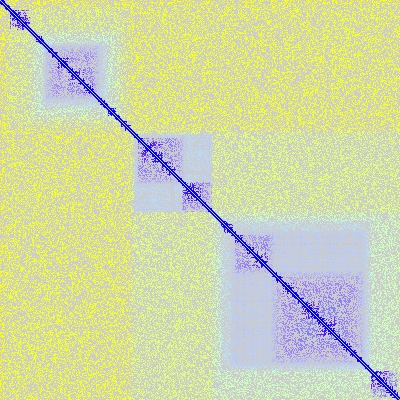}
\hskip0.1cm
\includegraphics[trim=0cm 0cm 0cm 0cm, clip, width=.3\textwidth]{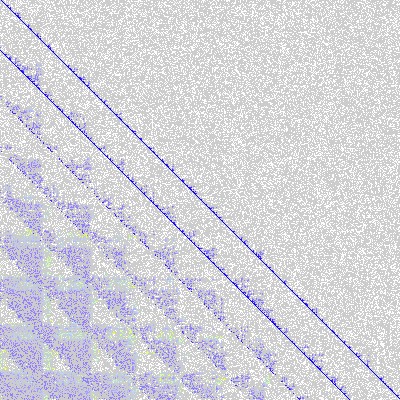}
\hskip0.1cm
\includegraphics[trim=0cm 0cm 0cm 0cm, clip, width=.3\textwidth]{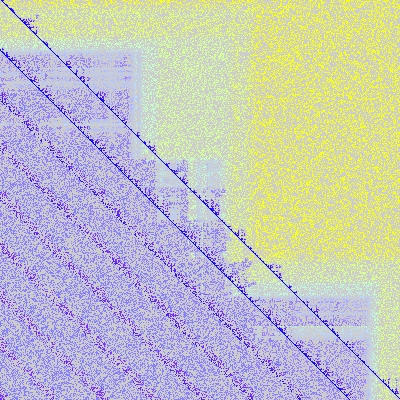}
\caption{Left: nucleation in $K_4$-percolation 
with $\po=0.39$.  Middle and right: 
illustration of Open Problem~\ref{con-oriented}, with $r=50$, and 
respective probabilities  
$\po=0.3$, $\po=0.37$.  In all figures, $n=400$ and the coloring scheme is similar to the one in Fig.~\ref{fig-tra}.}
\label{fig-op}
\end{figure}

Finally, we return to the transitive closure of oriented graphs, with $\Go$ 
the oriented \ER graph with probability $\po$ of edges.  
To understand oriented initial graphs, which are not covered
by Theorem~\ref{supercr-thm-R}, we may for example 
assume 
that 
$G_0$ is the oriented graph on $[n]$ with 
edges $1\er 2\er\cdots \er n$ and 
$1\el (1+r)\el (1+2r)\el \cdots 
\el (n-r)\el n$, where the range $r$ of leftward edges 
may grow with $n$. It is not difficult to see that 
$p_c$ is bounded away from $0$ when $r$ increases linearly 
with $n$, and that, by Theorem~\ref{supercr-thm-R},
 $p_c= (\log n)^{-1/2+o(1)}$ when $r$ is bounded.

\begin{openproblem}
\label{con-oriented} What is the asymptotic 
behavior of $p_c$, when $1\ll r\ll n$?
\end{openproblem}

These dynamics are illustrated in the middle and right panels of Fig.~\ref{fig-op}, which suggest that the most likely scenario
for saturation is through the early occupation of leftward edges 
whose lengths are multiples of $r$.


\makeatletter
\renewcommand\@biblabel[1]{#1.}
\makeatother


\begin{thebibliography}{10}

\bibitem{AL}
M.~Aizenman and J.~L. Lebowitz, \emph{Metastability effects in bootstrap
  percolation}, J. Phys. A \textbf{21} (1988), no.~19, 3801--3813.

\bibitem{AK}
O.~Angel and B.~Kolesnik, \emph{Sharp thresholds for contagious sets in random
  graphs}, Ann. Appl. Probab. \textbf{28} (2018), no.~2, 1052--1098.

\bibitem{BBDM}
J.~Balogh, B.~Bollob\'{a}s, H.~Duminil-Copin, and R.~Morris, \emph{The sharp
  threshold for bootstrap percolation in all dimensions}, Trans. Amer. Math.
  Soc. \textbf{364} (2012), no.~5, 2667--2701.

\bibitem{BBM}
J.~Balogh, B.~Bollob\'{a}s, and R.~Morris, \emph{Graph bootstrap percolation},
  Random Structures Algorithms \textbf{41} (2012), no.~4, 413--440.

\bibitem{BBMR12}
J.~Balogh, B.~Bollob\'{a}s, R.~Morris, and O.~Riordan, \emph{Linear algebra and
  bootstrap percolation}, J. Combin. Theory Ser. A \textbf{119} (2012), no.~6,
  1328--1335.

\bibitem{BHJ}
A.~D. Barbour, L.~Holst, and S.~Janson, \emph{Poisson approximation}, Oxford
  Studies in Probability, vol.~2, The Clarendon Press, Oxford University Press,
  New York, 1992, Oxford Science Publications.

\bibitem{BK20}
Z.~Bartha and B.~Kolesnik, \emph{Weakly saturated random graphs}, Random
  Structures Algorithms, to appear, available online at
  \url{https://doi.org/10.1002/rsa.21210}.

\bibitem{Bol}
B.~Bollob\'{a}s, \emph{Weakly {$k$}-saturated graphs}, Beitr\"{a}ge zur
  {G}raphentheorie ({K}olloquium, {M}anebach, 1967), Teubner, Leipzig, 1968,
  pp.~25--31.

\bibitem{BPRS17}
B.~Bollob\'{a}s, M.~Przykucki, O.~Riordan, and J.~Sahasrabudhe, \emph{On the
  maximum running time in graph bootstrap percolation}, Electron. J. Combin.
  \textbf{24} (2017), no.~2, Paper No. 2.16, 20.

\bibitem{BRSS}
B.~Bollob\'{a}s, O.~Riordan, E.~Slivken, and P.~Smith, \emph{The threshold for
  jigsaw percolation on random graphs}, Electron. J. Combin. \textbf{24}
  (2017), no.~2, Paper 2.36, 14.

\bibitem{BCDS}
C.~D. Brummitt, S.~Chatterjee, P.~S. Dey, and D.~Sivakoff, \emph{Jigsaw
  percolation: what social networks can collaboratively solve a puzzle?}, Ann.
  Appl. Probab. \textbf{25} (2015), no.~4, 2013--2038.

\bibitem{CLR79}
J.~Chalupa, P.~L. Leath, and G.~R. Reich, \emph{Bootstrap percolation on a
  {B}ethe lattice}, J. Phys. C \textbf{21} (1979), L31--L35.

\bibitem{CKM}
O.~Cooley, T.~Kapetanopoulos, and T.~Makai, \emph{The sharp threshold for
  jigsaw percolation in random graphs}, Adv. in Appl. Probab. \textbf{51}
  (2019), no.~2, 378--407.

\bibitem{Dur}
R.~Durrett, \emph{Oriented percolation in two dimensions}, Ann. Probab.
  \textbf{12} (1984), no.~4, 999--1040.

\bibitem{FKG}
C.~M. Fortuin, P.~W. Kasteleyn, and J.~Ginibre, \emph{Correlation inequalities
  on some partially ordered sets}, Comm. Math. Phys. \textbf{22} (1971),
  89--103.

\bibitem{FK}
E.~Friedgut and G.~Kalai, \emph{Every monotone graph property has a sharp
  threshold}, Proc. Amer. Math. Soc. \textbf{124} (1996), no.~10, 2993--3002.

\bibitem{GH}
J.~Gravner and A.~E. Holroyd, \emph{Polluted bootstrap percolation with
  threshold two in all dimensions}, Probab. Theory Related Fields \textbf{175}
  (2019), no.~1-2, 467--486.

\bibitem{GHS}
J.~Gravner, A.~E. Holroyd, and D.~Sivakoff, \emph{Polluted bootstrap
  percolation in three dimensions}, Ann. Appl. Probab. \textbf{31} (2021),
  no.~1, 218--246.

\bibitem{GM}
J.~Gravner and E.~McDonald, \emph{Bootstrap percolation in a polluted
  environment}, J. Statist. Phys. \textbf{87} (1997), no.~3-4, 915--927.

\bibitem{GS}
J.~Gravner and D.~Sivakoff, \emph{Nucleation scaling in jigsaw percolation},
  Ann. Appl. Probab. \textbf{27} (2017), no.~1, 395--438.

\bibitem{GKP17}
K.~Gunderson, S.~Koch, and M.~Przykucki, \emph{The time of graph bootstrap
  percolation}, Random Structures Algorithms \textbf{51} (2017), no.~1,
  143--168.

\bibitem{Hol}
A.~E. Holroyd, \emph{Sharp metastability threshold for two-dimensional
  bootstrap percolation}, Probab. Theory Related Fields \textbf{125} (2003),
  no.~2, 195--224.

\bibitem{RG}
S.~Janson, T.~{\L}uczak, and A.~Rucinski, \emph{Random graphs},
  Wiley-Interscience Series in Discrete Mathematics and Optimization,
  Wiley-Interscience, New York, 2000.

\bibitem{Karp}
R.~M. Karp, \emph{The transitive closure of a random digraph}, Random
  Structures Algorithms \textbf{1} (1990), no.~1, 73--93.

\bibitem{Kol}
B.~Kolesnik, \emph{The sharp ${K}_4$-percolation threshold on the
  {E}rd{\H{o}}s--{R}{\'e}nyi random graph}, Electron. J. Probab. \textbf{27}
  (2022), Paper No. 13, 23.

\bibitem{PRK75}
M.~Pollak and I.~Riess, \emph{Application of percolation theory to 2d-3d
  {H}eisenberg ferromagnets}, Physica Status Solidi (b) \textbf{69} (1975),
  no.~1, K15--K18.

\bibitem{BK85}
J.~van~den Berg and H.~Kesten, \emph{Inequalities with applications to
  percolation and reliability}, J. Appl. Probab. \textbf{22} (1985), no.~3,
  556--569.

\end{thebibliography}

\providecommand{\bysame}{\leavevmode\hbox to3em{\hrulefill}\thinspace}
\providecommand{\MR}{\relax\ifhmode\unskip\space\fi MR }
\providecommand{\MRhref}[2]{%
  \href{http://www.ams.org/mathscinet-getitem?mr=#1}{#2}
}
\providecommand{\href}[2]{#2}

\end{document}